\documentclass[11pt]{amsart}
\usepackage{amssymb}
\usepackage{amsmath}
\usepackage{tikz}
\usetikzlibrary{graphs, decorations.pathmorphing, decorations.pathreplacing}
\usepackage{hyperref, cleveref}
\usepackage{geometry}

\hypersetup{
	pdfstartview={XYZ null null 1.00}, 
	pdfpagemode=UseNone, 
	colorlinks,
	breaklinks, 
	linkcolor=blue,
	urlcolor=blue, 
	anchorcolor=blue,
	citecolor=blue
}

\newtheorem{theorem}{Theorem}[section]
\newtheorem{question}{Question}
\newtheorem{lemma}[theorem]{Lemma}
\newtheorem{prop}[theorem]{Proposition}
\newtheorem{cor}[theorem]{Corollary}

\newtheorem{claim}[theorem]{Claim}

\theoremstyle{definition}
\newtheorem{definition}[theorem]{Definition}

\newcommand{\dom}{\textnormal{dom}}

\newcommand{\inv}{^{-1}}
\newcommand{\Z}{\mathbb{Z}}

\newcommand{\N}{\mathbb{N}}

\def\acts{\curvearrowright}

\newcommand{\res}{\upharpoonright}

\newcommand{\sch}{\textnormal{Sch}}
\newcommand{\cay}{\textnormal{Cay}}

\title{Borel Homomorphisms from Forests to Kneser Graphs}
\author{Felix Weilacher}

\begin{document}

\begin{abstract}
We answer a recent question of Cs\'oka and Vidny\'anszky \cite{csoka_vidnyanszky} and give an alternate proof of one of their results. The subject of both is which finite graphs admit factor of i.i.d. homomorphisms from the 3-regular tree. We then give yet another proof of the result in the Borel setting which leads to the following: For each $d > 2$ and $k \in \N$, there is a Borel hyperfinite $d$-regular forest $G$ and a finite graph with chromatic number $k$ $H$ so that $G$ does not admit a Borel homomorphism to $H$. All of this is tied together by a focus on the case when the target graph $H$ is a (subgraph of a) Kneser graph.
\end{abstract}

\maketitle

\section{Introduction}

This paper is about the Borel and measurable combinatorics of 3-regular forests, or more generally acyclic (regular) graphs. It is largely a response to a recent paper of Cs\'oka and Vidny\'anszky \cite{csoka_vidnyanszky}, and like them we focus on a specific family of graph combinatorial problems: finding homomorphisms to fixed finite graphs. These problems are natural and generalize the problem of finding proper colorings; for $k \in \N$ a proper $k$-coloring is the same as a homomorphism to $K_k$, the complete graph on $k$ vertices.

\begin{definition}
\
    \begin{itemize} 
        \item A \textit{Borel graph} is a graph $G$ such that $V(G)$ is a standard Borel space and $E(G) \subseteq V(G)^2$ is Borel in the product space. 
        \item Let $G$ and $H$ be graphs. A \textit{homomorphism} from $G$ to $H$ is a function $f : V(G) \to V(H)$ such that for all $(x,y) \in E(G)$, $(f(x),f(y)) \in E(H)$. 
        \item When $G$ and $H$ are Borel graphs, a homomorphism from $G$ to $H$ is called \textit{Borel} if it is Borel as a function on vertices. 
    \end{itemize}
\end{definition}

We will implicitly consider finite graphs as Borel graphs by giving the vertex set the discrete Borel structure. The following is our basic question of interest. 

\begin{question}\label{q:borel_hom}
    Which finite graphs $H$ have the following property: For any 3-regular acyclic Borel graph $G$, $G$ admits a Borel homomorphism to $H$. 
\end{question}

Though there has been interesting progress for some $H$, the question in general seems difficult. For instance, we do not even know whether the set of such $H$ is decidable, although this seems unlikely. If the degree 3 is replaced with 2, the question has an easy answer: $H$ has the property if and only if it has an odd cycle (See for example \cite{GR_paths}). For degrees larger than 3 the question is equally interesting. 

Part of the interest in Borel combinatorics of regular forests comes from its connection with fiid combinatorics on regular trees, a well-studied topic in probability. We will give a definition in Section \ref{subsec:shift}, but for now it suffices to say that for $H$ a finite graph and $d \in \N$, the $d$-regular tree, denoted $T_d$, admits an ``fiid homomorphism to $H$'' if and only if a certain fixed Borel $d$-regular forest on a standard probability space admits a measurable homomorphism to $H$. Thus, if $H$ has a positive answer for Question \ref{q:borel_hom}, it also does for the following equally important question. 

\begin{question}
    Which finite graphs $H$ admit fiid homomorphisms from $T_3$?
\end{question}

Cs\'oka and Vidny\'anszky made the following contribution to these questions.

\begin{theorem}[\cite{csoka_vidnyanszky}]\label{thm:cv_girth}
    For every $r \in \N$ there exists $C_r \in \N$ such that if a finite $r$-regular graph $H$ has girth at least $C_r$, then there is no fiid homomorphism from $T_3$ to $H$. 
\end{theorem}

\begin{cor}[\cite{csoka_vidnyanszky}]\label{cor:cv_chromatic}
    There are finite graphs $H$ of arbitrarily large chromatic number which do not admits fiid homomorphisms from $T_3$. 
\end{cor}

\begin{proof}
    Bollob\'as proved \cite{bollobas1978chromatic} that for every $r$ there are $r$-regular graphs with arbitrarily large girth and chromatic number $\Omega(r/\log r)$. 
\end{proof}

We remark that the proof of Theorem \ref{thm:cv_girth} still goes through if ``girth'' is replaced by ``odd girth''. 
By the comment above these results also apply to Question \ref{q:borel_hom}. Theorem \ref{thm:cv_girth} also leads them to ask the following question. 

\begin{question}[\cite{csoka_vidnyanszky}]\label{q:cv_triangle}
    Is there a finite triangle free graph $H$ which admits an fiid homomorphism from $T_3$?
\end{question}

The first goal of this paper is to point out that by considering the case where $H$ is a \textit{Kneser graph}, we can obtain an alternate proof of Corollary \ref{cor:cv_chromatic} and a positive answer to Question \ref{q:cv_triangle} using already-known results about these graphs. In fact for the latter we can prove the following:

\begin{theorem}\label{thm:easy_girth}
    For any $d,g \in \N$, there is a finite graph $H$ with odd girth at least $g$ with the property that any $d$-regular acyclic Borel graph $G$ admits a Borel homomorphism to $H$. 
\end{theorem}

``Borel'' can be upgraded to ``continuous'' in this result provided the appropriate topological assumptions on $G$ for this to make sense. We do not know if ``odd girth'' can be strengthened to ``girth''. See Question \ref{q:odd_girth}.

The second goal of this paper is to replicate Corollary \ref{cor:cv_chromatic} in the Borel setting for \textit{hyperfinite graphs}. 

\begin{definition}
\
    \begin{itemize}
        \item A graph is called \textit{component finite} if each of its connected components is finite.
        \item A Borel graph $G$ is called \textit{hyperfinite} if there is a sequence $G_0 \subseteq G_1 \subseteq \cdots$ of component finite Borel graphs, each with vertex set $V(G)$, with $\bigcup_n E(G_n) = E(G)$. 
    \end{itemize}
\end{definition}

It is easy to see that hyperfiniteness of a Borel graph $G$ depends only on the equivalence relation whose classes are the connected components of $G$. In fact the motivation for studying hyperfinite Borel graphs comes from the theory of so-called \textit{countable Borel equivalence relations} (CBERs), where hyperfiniteness is an essential notion with many notorious open problems attached. 

\begin{question}\label{q:hyperfinite_hom}
    Which finite graphs $H$ have the following property: For any 3-regular \textbf{hyperfinite} acyclic Borel graph $G$, $G$ admits a Borel homomorphism to $H$. 

    Is the set of such $H$ the same as the set from Question \ref{q:borel_hom}?
\end{question}

A negative answer to the second part of this question would be quite significant as it would give us a seemingly new tool for proving that a (3-regular acyclic) graph is not hyperfinite. Our current set of tools for proving that CBERs are non-hyperfinite is quite limited, and expanding it will be necessary for making progress on the aforementioned open questions about hyperfiniteness. Unfortunately, Conley et al. \cite{hyperfinite.acyclic} have shown that when $H$ is a complete graph, the answers to Questions \ref{q:borel_hom} and \ref{q:hyperfinite_hom} are the same. That is, the proper coloring problem cannot be used to distinguish hyperfinite forests. 

This motivates the extension of some of the results from \cite{csoka_vidnyanszky} to the hyperfinite setting.

\begin{theorem}\label{thm:hyperfinite_girth}
    Let $d,l \in \N$, $d > 1$. Suppose $H$ is a graph with $|V(H)| \leq d\cdot (d-1)^{l}$ and odd girth greater than $2l + 1$. 
    Then there is a hyperfinite $d$-regular Borel forest with no Borel homomorphism to $H$. 
\end{theorem}

\begin{cor}\label{cor:hyperfinite_chromatic}
    For every $d \in \N$, $d > 2$, there are finite graphs of arbitrarily large chromatic number which do not admit Borel homomorphisms from an arbitrary hyperfinite $d$-regular acyclic Borel forest. 
\end{cor}

Of course, to prove the corollary from the theorem we will need to find a finite graph $H$ with the right mix of odd girth, number of vertices, and chromatic number. For this we will use certain critical subgraphs of the Kneser graphs already studied in combinatorics, the so-called \textit{Schrijver graphs} \cite{schrijver_graph}. See Section \ref{sec:k_chromatic}. 

Observe that setting $l = 0$ and $H = K_d$ in Theorem \ref{thm:hyperfinite_girth} recovers the result from \cite{hyperfinite.acyclic}: there is a hyperfinite $d$-regular acyclic Borel forest with no Borel proper $d$-coloring. 

These results touch on the relationship between measure and hyperfiniteness. When null-sets are ignored, hyperfiniteness becomes much more well understood. Among other things, both parts of the analogue of Question \ref{q:hyperfinite_hom} are known; it follows from \cite{BCGGRV} (see also \cite{conley.miller.toast}) that $H$ has a measurable homomorphism from any hyperfinite acyclic regular Borel graph on a standard probability space if and only if $H$ has an odd cycle. Thus measure cannot be used to even make a dent in Corollary \ref{cor:hyperfinite_chromatic}.

In place of measure, like the authors in \cite{hyperfinite.acyclic}, we will turn in Section \ref{sec:main} to the determinacy method of Marks \cite{Marks}. We follow the newer approach of Brandt et al. \cite{hom_graph}, who gave an alternate proof of the main theorem from \cite{hyperfinite.acyclic}. 

%Conley and Miller showed \cite{conley.miller.toast} that hyperfinite Borel forests always have measurable 3-colorings, while Conley and Kechris showed \cite{conley2013measurable} that without hyperfiniteness this fails for $d$-regular forests for $d$ large enough. For $d = 3$ one can also distinguish hyperfinite from general forests using large enough odd cycles thanks to results from \cite{BCGGRV} and Theorem \ref{thm:fiid_ind_numbers} below. 

\section{Preliminaries}

\subsection{Bernoulli shifts and Factors of i.i.d. variables}\label{subsec:shift}

In this section we define the basic Borel graphs which most of our questions are about. 

\begin{definition}
    \ 
    \begin{itemize}
        \item     Let $\Gamma$ be a countable group and $X$ a set. The \textit{Bernoulli shift} of $\Gamma$ with base space $X$ is the left shift action $\Gamma \acts X^\Gamma$: 
    \[ \gamma \cdot x = \delta \mapsto x(\gamma\inv \delta). \]
        \item $F(\Gamma,X)$ denotes the restriction of this action to the \textit{free part}, the set of points with trivial stabilizer. 
        \item Let $\Gamma \acts X$ be an action, and $S \subseteq \Gamma$ symmetric. The \textit{Schreier graph} of this action with respect to $S$, denoted $\sch(\Gamma,S,X)$, is the graph with vertex set $X$ where two vertices $x,y \in X$ are adjacent if $\exists \gamma \in S, \gamma x = y$.  
        \item $\cay(\Gamma,S) := \sch(\Gamma,S,\Gamma)$, where the action $\Gamma \acts \Gamma$ is by \textbf{right} multiplication.
        \item For $X$ a set, $S(\Gamma,S,X) := \sch(\Gamma,S,F(\Gamma,X))$.
    \end{itemize}

\end{definition}

That is, $S(\Gamma,S,X)$ is the Schreier graph of the free part of the Bernoulli shift of $\Gamma$ with base space $X$. When $X$ is a standard Borel space, standard probability space, or topological space, this structure is inherited by $X^\Gamma$ and then $F(\Gamma,X)$. For instance, if $X$ is a standard Borel space then $S(\Gamma,S,X)$ is a Borel graph. 
$F(\Gamma,2^\omega)$ has the important property that any free Borel $\Gamma$-action on a standard Borel space admits a Borel equivariant injection into it. Thus the Schreier graph of any such action with respect to a given $S \subseteq \Gamma$ admits an injective Borel homomorphism to $S(\Gamma,S,2^\omega)$. Thus Borel graph homomorphism problems often reduce to consideration of these graphs.

For much of this paper, following \cite{BCGGRV, hom_graph}, we will use the Cayley graph of the group $\Z_2^{* d} := \langle a_0,\ldots,a_{d-1} \mid \forall i, a_i^2 = 1 \rangle $ with respect to the generating set $S_d := \{a_0,\ldots,a_{d-1}\}$ as a stand-in for $T_d$. Of course, these are the same graph, but the former has a canonical $d$-edge coloring which can make a difference in some settings.

For the definition of fiid processes on $T_d$ we will need to lose these edge colors, so we need to define a Bernoulli shift for this graph without any group action.

\begin{definition}
    \ 
    \begin{itemize}
        \item Let $o \in V(T_d)$, so that $(T_d,o)$ is the rooted $d$-regular tree. Let $X(T_d)$ denote the set of copies of $(T_d,o)$ with each vertex labeled by an element of $2^\omega$, where to copies are considered the same if there is an isomorphism between them preserving the labels. 
        \item Let $S'(T_d)$ be the graph on $X(T_d)$ in which two labeled trees $(T_d,o,x)$ and $(T_d,o,y)$ are adjacent if there is a vertex $v \in N(o)$ such that $(T_d,v,x) \cong (T_d,o,y)$. 
        \item Let $F(T_d) \subseteq X(T_d)$ consist of the labeled trees whose connected components in $S'(T_d)$ are $d$-regular trees. Let $S(T_d) = S'(T_d) \res F(T_d)$. 
    \end{itemize}
\end{definition}

It is easy to see that $S(T_d)$ is universal for Borel $d$-regular forests in the same way $F(\Gamma,2^\omega)$ and $S(\Gamma,-,2^\omega)$ are for Borel $\Gamma$-spaces.

Like $F(\Gamma,2^\omega)$, $F(T_d)$ inherits a natural probability measure from the coin flip measure on $2^\omega$. 

\begin{definition}
    A \textit{factor of i.i.d.} (fiid) labeling of $T_d$ is a measurable function $F(T_d) \to \Lambda$ for some finite set $\Lambda$. An fiid homomorphism of $T_d$ to a finite graph $H$ is a measurable homomorphism $S(T_d) \to H$. 
\end{definition}

The interpretation of this is that a measurable function $c : F(T_d) \to \Lambda$ 
is the same as a random variable taking values in $\Lambda^{T_d}$ (modulo automorphisms) which can be obtained as a factor of the random variable taking values in $(2^\omega)^{T_d}$ in which each vertex independently randomly chooses a label from $2^\omega$ according to the coin flip measure. 

\subsection{Kneser graphs and fractional colorings}

\begin{definition}[{See \cite[Chapter 3]{scheinerman_ullman}}]
    Let $k \le n \in \N$. The \textit{Kneser graph} $K(n,k)$ is the graph with vertex set $\binom{n}{k}$, the set of $k$-element subsets of $n$, where two sets are adjacent if they are disjoint. 
\end{definition}

When $n < 2k$ $K(n,k)$ contains no edges, so one typically assumes $n \geq 2k$.
These graphs are well-studied in combinatorics, partly because of the natural connection they have with fractional chromatic numbers.
See for instance \cite{scheinerman_ullman}. 

\begin{definition}
    \ 
    \begin{itemize}
        \item Let $k \leq n \in \N$. Let $G$ be a graph. A $k$-fold $n$-coloring of $G$ is a sequence $I_0,\ldots,I_{n-1} \subseteq V(G)$ of independent sets so that each vertex is in exactly $k$ sets.
        \item The \textit{fractional chromatic number} of $G$, denoted $\chi^*(G)$, is 
        \[ \chi^*(G) := \inf \{ \frac{n}{k} \mid G \textnormal{ admits a } k\textnormal{-fold } n \textnormal{ coloring}\}.  \]
    \end{itemize}
\end{definition}

The point is that a $k$-fold $n$-coloring of $G$ is, plainly, the same as a homomorphism to $K(n,k)$. This generalizes the fact that proper $n$-colorings are the same as homomorphisms to $K_n = K(n,1)$. 

It is common to lower bound the chromatic number of a graph by upper bounding the size of its independent sets. This bound also applies to the fractional chromatic number. 

\begin{definition}
    Let $G$ be a finite graph. $\alpha(G) := m/|V(G)|$, where $m$ is the maximal size of an independent set. This is called the \textit{independence ratio}.
\end{definition}

\begin{lemma}[{\cite[Proposition 3.1.1]{scheinerman_ullman}}]\label{lem:fractional_chrom_vs_ind}
    Let $G$ be a finite graph. $\chi^*(G) \geq 1/\alpha(G)$. When $G$ is vertex transitive, this is an equality.
\end{lemma}

% When $G$ is a Cayley graph, this becomes an equality. 

% \begin{lemma}[\todo{cite}]\label{lem:fractional_cayley}
%     Let $\Gamma$ be a finite group and $S \subseteq \Gamma$ symmetric. $\chi^*(\cay(\Gamma,S)) = 1/\alpha(\cay(\Gamma,S))$.
% \end{lemma}

Since in this paper we work with Borel graphs, to get a useful analogue of this Lemma we need to replace counting with a measure. 

\begin{definition}
    Let $G$ be a Borel graph and $\mu$ a Borel probability measure on $V(G)$.
    \begin{itemize}
        \item $\alpha_{\mu}(G)$ is the supremum of the measures $\mu(I)$ for $I \subseteq V(G)$ a Borel (equivalently, measurable) independent set. This is called the \textit{measurable independence ratio}.
        \item A \textit{Borel $k$-fold $n$-coloring} of $G$ is a Borel homomorphism from $G$ to $K(n,k)$. Likewsie for a $\mu$-measurable $k$-fold $n$-coloring. The Borel and $\mu$-measurable fractional chromatic numbers of $G$, denoted $\chi^*_B(G)$ and $\chi^*_\mu(G)$ respectively, are defined accordingly. 
    \end{itemize}
\end{definition}

\begin{lemma}\label{lem:meas_ind_leq}
    Let $G$ be a Borel graph and $\mu$ a Borel probability measure on $V(G)$.
    $\chi_B^*(G) \geq \chi_\mu^*(G) \geq 1/\alpha_\mu(G)$.    
\end{lemma}

The study of fractional chromatic numbers in the context of descriptive combinatorics was initiated by Meehan \cite{meehan}. Bernshteyn then proved an analogue of the second part of Lemma \ref{lem:fractional_chrom_vs_ind}.

\begin{theorem}[\cite{bernshteyn_fractional}]\label{thm:meas_ind_eq}
    Let $\Gamma$ be a countable group and $S \subseteq \Gamma$ finite and symmetric. Let $\mu$ be the usual measure on $F(\Gamma,2^\omega)$. 
    \[ \chi_B^*(S(\Gamma,S,2^\omega)) = 1/\alpha_\mu(S(\Gamma,S,2^\omega)). \]
\end{theorem}

We emphasize that on the LHS one has the Borel, not $\mu$-measurable, fractional chromatic number. $\alpha_\mu(S(\Gamma,S,2^\omega))$ is, in the language of the previous section, the ``fiid independence number'' of $\cay(\Gamma,S)$. It is not hard to check that Bernshteyn's result also applies in situations like the Bernoulli shift of $T_d$. 

\begin{prop}\label{prop:meas_ind_eq_Td}
    Let $d \in \N$, $d \geq 2$. Let $\mu$ be the usual measure on $F(T_d)$. 
    \[ \chi_B^*(S(T_d)) = 1/\alpha_\mu(S(T_d)). \]
\end{prop}

A lot of effort has gone into computing $\alpha_\mu(S(T_d))$, the fiid independence number of $T_d$, both asymptotically and for small values. We will only need the latter.

\begin{theorem}[\cite{csoka2015invariant,mckay1987ind}]\label{thm:fiid_ind_numbers}
 $0.4361 \leq \alpha_\mu(S(T_3)) \leq 0.45537$
\end{theorem}

%%%%%%%%%%%%%%%%%%%%%%%%%%%%%%%

\subsection{Chromatic numbers of Kneser graphs}\label{sec:k_chromatic}

To understand the relationship between fractional and ordinary colorings, it is important to compute the chromatic numbers of the Kneser graphs. This turns out to be nontrivial, and was first done by Lov\'asz using topological methods.

\begin{theorem}[\cite{lovasz1978kneser}]\label{thm:Lovasz}
    Let $k,n \in \N$ with $n \geq 2k$. $\chi(K(n,k)) = n - 2k + 2.$
\end{theorem}

On the other hand, unsurprisingly (though still not trivially), $\chi^*(K(n,k)) = n/k$ \cite{erdos_ko_rado}.

We now have enough information to give our alternate proof of Corollary \ref{cor:cv_chromatic} and our positive answer to Question \ref{q:cv_triangle}. 

\begin{proof}[Proof of Corollary \ref{cor:cv_chromatic}]
    By Theorem \ref{thm:fiid_ind_numbers} and Lemma \ref{lem:meas_ind_leq}, $$\chi_{\mu}^*(S(T_d)) \geq 1/0.45537 > 2 + \frac{1}{6}.$$ 
    Therefore $T_d$ has no FIID homomorphism to $K(2k + \frac{k}{6},k)$ for any $k \in 6\N$. By Theorem \ref{thm:Lovasz} the chromatic number of this graph is $2 + \frac{k}{6}$, so since $k$ is arbitrary we are done. 
\end{proof}

\begin{proof}[Proof of positive answer to Question \ref{q:cv_triangle}]
    By Theorem \ref{thm:fiid_ind_numbers} and Proposition \ref{prop:meas_ind_eq_Td}, $$\chi_B^*(S(T_d)) \leq 1/0.4361 < 2 + \frac{1}{3}.$$ 
    Therefore there exists $k \in 3\N$ such that $S(T_d)$ admits a Borel homomorphism to $K(2k + \frac{k}{3},k)$. Clearly $K(n,k)$ is triangle free iff $n < 3k$, so we are done. 
\end{proof}

The observation about when $K(n,k)$ contains a triangle easily extends to the following, which we will need shortly. When $n = 2k$ $K(n,k)$ is a matching, so we further restrict our attention to the case $n > 2k$. 

\begin{lemma}[{\cite[Corollary 2.11]{poljak1987maximum}}]\label{lem:kneser_odd_girth}
    Let $n,k \in \N$ with $n > 2k$. The odd girth of $K(n,k)$ is 
    \[ 1 + 2 \left\lceil \frac{k}{n - 2k} \right\rceil \].
\end{lemma}

On the other hand $K(n,k)$ contains a 4-cycle whenever $n \geq 2k+2$. 

When $n \sim (2+\epsilon)k$ for $\epsilon > 0$ small and $k$ large, the above results say that $K(n,k)$ has both large odd girth and large chromatic number. Unfortunately $K(n,k)$ has too many vertices, (around $2^n$) to deduce Corollary \ref{cor:hyperfinite_chromatic} from Theorem \ref{thm:hyperfinite_girth}.
Fortunately, it is known that there are \textit{critical subgraphs} of $K(n,k)$, i.e, minimal subgraphs with the same chromatic number, with much fewer vertices. 

\begin{definition}[\cite{schrijver_graph}]
    Let $n,k \in \N$ with $n \geq 2k$. The \textit{Schrijver graph} $K'(n,k)$ is the induced subgraph of $K(n,k)$ on the set of sets $A \in \binom{n}{k}$ with the property that no two elements of $A$ have a difference of 1 mod $n$.     
\end{definition}

\begin{theorem}[\cite{schrijver_graph}]\label{thm:schrijver_chromatic}
    $\chi(K'(n,k)) = \chi(K(n,k))$ $(= n-2k+2)$
\end{theorem}

Finally we will need the following computation

\begin{lemma}\label{lem:schrijver_cardinality}
    Fix $m \in \N$. $|V(K'(2k+m,k))| = O(k^{m})$.
\end{lemma}
\begin{proof}
    In general, $|V(K'(n,k))| = \frac{n}{k} \binom{n-k-1}{k-1}$. This is well known, but for instance, one can specify a vertex $A$ by choosing a ``starting element'' $x \in n$, then writing the sequence in $\N^{k-1}$ of lengths of gaps between each pair of consecutive elements of $A$ starting from $x$. Each gap must have positive length and the lengths must add to at most $n-k-1$, so stars and bars show that there are $\binom{n-k-1}{k-1}$ such sequences. Finally, any element of $A$ can be the starting element, so we have overcounted by a factor of $k$.
    
    Now 
    \begin{gather*}
        \frac{2k + m}{k} \binom{k+m-1}{k-1} = O(1) \binom{k+m-1}{m} \leq O(1)(k+m)^m = O(k^m).
    \end{gather*}
\end{proof}

\begin{proof}[Proof of Corollary \ref{cor:hyperfinite_chromatic} from Theorem \ref{thm:hyperfinite_girth}]
    Fix $d,m \in \N$ with $d > 2$. We will take $H = K'(2k+m,k)$ for some large $k$. This always has chromatic number $m+2$ by Theorem \ref{thm:schrijver_chromatic}, so since $m$ is arbitrary we will be done if we can check the hypotheses of Theorem \ref{thm:hyperfinite_girth}. 

    Let $l = (k/m) - 1$, so that by Lemma \ref{lem:kneser_odd_girth} $K(2k+m,k)$, and therefore $K'(2k+m,k)$, has odd girth greater than $2l+1$. Now $|V(K'(2k+m,k))|$ is at most polynomial in $k$ by Lemma \ref{lem:schrijver_cardinality}, so for large enough $k$ it will be less than $(d-1)^{k/m} \leq d \cdot (d-1)^l$, as desired. 
    Note the use of $d > 2$ here. 
\end{proof}

\section{Proof of Theorem \ref{thm:easy_girth}}

In this short section we prove Theorem \ref{thm:easy_girth}. The proof is already implicit in Bernshteyn's proof of Theorem \ref{thm:meas_ind_eq} \cite{bernshteyn_fractional} (recall that we used this in our answer to Question \ref{q:cv_triangle}), where the homomorphism to the Kneser graph factors through the sort of homomorphism we will describe. In fact, Bernshteyn and independently Seward have shown that all continuous homomorphisms from $S(T_d)$ to finite graphs factor through these sorts of homomorphisms \cite{bernshteyn2023cts}. 

Fix $d$ and $g > 0$. Let $o \in V(T_d)$ and $T = T_d \res B(o,g)$. Let $N = d^{2g}$. 
$H$ will a graph whose vertex set is the set of injective functions $B(o,g) \to N$ modulo automorphisms of $T$. We think of these as labelings of $T$. Two such labelings $f_0$ and $f_1$ will be adjacent iff there is a labeing $F : V(T) \to N$ and adjacent vertices $(v_0,v_1) \in E(T_d)$ so that for each $i \in 2$, $(T_d \res B(v_i,g), F \res B(v_i,g)) \cong (T,f_i)$.
There is an alternate description of the edge set which may be helpful: Identifying $N$ with some set of $N$ points in $2^\omega$, call a \textit{lift} of $f \in H_i$ a labeling $x \in F(T_d)$ which extends $f$. Then $f_0,f_1 \in V(H)$ are adjacent iff they have lifts which are adjacent in $S(T_d)$. 

\begin{claim}
$H$ has odd girth $> g$. 
\end{claim}
\begin{proof}
    Let $l \leq g$ odd and suppose $f_0, \ldots, f_l = f_0$ is a cycle of length $l$ in $H$. WLOG $f_0(o) = 0$. By induction one can see that for each $i \leq l$, there is some $v_i \in B(o,g)$ with $f_i(v) = 0$, and, letting $d_i$ be the distance from $o$ to $v_i$, $d_{i+1} = d_i \pm 1$. In particular $d_i = i$ (mod 2). Since $l$ is odd, we then have $v_l \neq o$, but $f_l = f_0$, so this contradicts the injectivity of $f_0$. 
\end{proof}

In fact it is not hard to check that the odd girth of $H$ is precisely $2g + 1$. 

\begin{claim}
Let $G$ be a $d$-regular acyclic Borel graph. $G$ admits a Borel homomorphism to $H$. 
\end{claim}
\begin{proof}
    Let $G' = G^{\leq 2g}$ denote the Borel graph with $V(G') = V(G)$ and where two vertices are adjacent if their distance in $G$ is at most $2g$. Note that $G'$ is $d(d-1)^{2g-1}$-regular. Since this is strictly less than $N$, it follows from \cite{kst} $G'$ admits a Borel proper $N$-coloring, call it $f$.

    Now let $c : V(G) \to V(H)$ be the function sending a vertex $x$ to the labeling $f \res B(x,g)$. By definition of $G'$ and $f$, this labeling is injective, so it is indeed in $V(H)$. It follows from the definition of $H$ that $c$ is a homomorphism from $G$ to $H$, and it is clearly Borel. 
\end{proof}

This completes the proof of Theorem \ref{thm:easy_girth}. Note that $H$ has many 4-cycles. We do not know if these can be removed. 

\begin{question}\label{q:odd_girth}
    Is There a finite graph $H$ with girth at least 5 with the property that that any Borel 3-regular forest admits a Borel homomorphism to $H$?
\end{question}

% \section{Non-hyperfinite examples}

% In this section we will prove the weakening of Theorem \ref{thm:hyperfinite_girth} where the Borel forest is not required to be hyperfinite. 

% \begin{prop}
%     Let $d,l \in \N$, $d > 1$. Suppose $H$ is a graph with $|V(H)| \leq d \cdot (d-1)^l$ and odd girth greater than $2l + 1$. Then there is a $d$-regular Borel forest with no Borel homomorphism to $H$. 
% \end{prop}

% The proof of this proposition will contain the key combinatorial idea of the paper. 

% The case $l = 0$, $H = K_d$ is due to Marks \cite{Marks, marks2013short}, and we follow the structure of his proof. For instance, our counterexample will be the graph $S(\Z_2^{*d},S_d,\omega)$. 

\section{Proof of Theorem \ref{thm:hyperfinite_girth}}\label{sec:main}

We mimic the proof from \cite{hom_graph} of the case $l = 0$, $H = K_d$. 
Our counterexample graph $G$ will be the same as theirs, so we start by defining it.
Let $e : \omega \to d$ be any function with each fiber infinite. Let $s_n \in 2^n$ for $n \in \omega$ so that for each $i < d$, the set $\{s_n \mid e(n) = i\}$ is dense in $2^{<\omega}$. Recall the graph $\mathbb{G}_0$ on $2^\omega$ defined by Kechris, Solecki, and Todor\v{c}evi\'{c}. 

\begin{definition}[\cite{kst}]
    $E(\mathbb{G}_0) = \{ (s_n ^\frown \epsilon ^\frown z, s_n ^\frown (1-\epsilon) ^\frown z) \mid n \in \omega, \epsilon \in 2, z \in 2^\omega\}.$
\end{definition}

Following now \cite[Section 4.2]{hom_graph}, we define a $d$-edge labeling of this version of $\mathbb{G}_0$, and pick a subgraph on which the labeling is well-behaved. Observe that if $(x,y) \in E(\mathbb{G}_0)$ then $x$ and $y$ differ in exactly one bit. 

\begin{definition}[\cite{hom_graph}]
\ 
\begin{itemize}
    \item Let $\alpha : E(\mathbb{G}_0) \to d$ be the following (symmetric) edge labeling: Given $(x,y) \in E(\mathbb{G}_0)$, let $n \in \omega$ be the unique $n$ with $x(n) \neq y(n)$. Then $\alpha(x,y) = e(n)$. 
    \item Let $\mathcal{H}$ be the induced subgraph of $\mathbb{G}_0$ on the set 
    \begin{align*} V(\mathcal{H}) 
    := \{x \in 2^\omega \mid \text{every vertex in the } \mathbb{G}_0 \text{ component of } x \\ \text{ is incident to infinitely many edges of every } \alpha \text{ label}.\}
    \end{align*}
\end{itemize}
\end{definition}

\begin{lemma}[Essentially {\cite[Claim 4.7]{hom_graph}}]\label{lem:edge_G0}
\ 
    \begin{enumerate}
        \item $V(\mathcal{H})$ is comeager in $2^\omega$
        \item Let $A \subseteq V(\mathcal{H})$ be Baire measurable and nonmeager. For each $i \in d$ there exist $x,x' \in A$ such that $(x,x')$ is an edge in $\mathcal{H}$ with label $i$. 
    \end{enumerate}
\end{lemma}

We can now define our counterexample $G$. Recall that $S_d = \{a_0,\ldots,a_{d-1}\}$ is our generating set for $\Z_2^{*d}$.

\begin{definition}[\cite{hom_graph}]
Let $G$ be the graph $\textnormal{Hom}^e(T_d,\mathcal{H})$, defined as follows: 
$V(G) \subseteq (2^\omega)^{\Z_2^{*d}}$ is the set of graph homomorphisms $z: \textnormal{Cay}(\Z_2^{*d},S_d) \to \mathcal{H}$ that preserve the edge labels. That is, so that for each $\gamma \in \Z_2^{*d}$ and $i \in d$, $\alpha(z(a_i),z(\gamma a_i)) = i$. 

Let $G$ be the induced subgraph of $S(\Z_2^{*d},S_d,2^\omega)$ on $V(G)$. 
\end{definition}

It is not hard to see that $V(G)$ is an invariant subset of $F(\Z_2^{*d},2^\omega)$, so the final operation here is well defined, and $G$ is acyclic and $d$-regular. It also turns out that $G$ is hyperfinite since $\mathcal{H}$ is \cite[Proposition 2.6]{hom_graph}. 

We now describe a method of constructing elements of $G$ that yields a natural bijection between $V(G)$ and $V(\mathcal{H}) \times \omega^{\Z_2^{*d} \setminus \{1\}}$. First, let $x \in V(\mathcal{H})$ and $i \in d$. Let $n_0^i < n_1^i < \cdots$ enumerate the elements $n$ of $e\inv(i)$ for which $s_{n} \subseteq x$. Then for each $m$, $x$ is adjacent to the real obtained by flipping the $n_m^i$-th bit of $x$ and the edge to this vertex has color $i$. Furthermore this is enumerates all such edges. Let $f_{i,m}(x)$ denote the neighbor corresponding to $m \in \omega$. We will need the following later.

\begin{claim}\label{c:continuous_inv}
    Each $f_{i,m}$ is a continuous involution of $V(\mathcal{H})$. 
\end{claim}
\begin{proof}
    Let $y = f_{i,m}(x)$. $x \res n_m^i = y \res n_m^i$, so for $n \leq n_m^i$ they agree on which $s_n$'s they have as prefixes. Thus $f_{i,m}(y) = x$. Furthermore, any $z$ with $z \res n_m^i = x \res n_m^i$ will for the same reason agree on which bit is flipped by $f_{i,m}$, proving continuity. 
\end{proof}

Now again let $x \in V(\mathcal{H})$, and let $y \in \omega^{\Z_2^{*d} \setminus \{1\}}$. We will use $y$ to define an element $z \in V(G)$ with $z(1) = x$. We define $z(\gamma)$ by induction on the word length of $\gamma$ in terms of the generating set $S_d$. Given $\gamma \neq 1$, write $\gamma = \delta a_i$ where $\delta \in \Z_2^{*d}$ is a shorter word, and define
$z(\gamma) = f_{i,y(\gamma)}(z(\delta))$. 
Write $z = \Phi(x,y)$. It is easy to see that $\Phi$ is a Borel bijection from $V(\mathcal{H}) \times \omega^{\Z_2^{*d} \setminus \{1\}}$ to $V(G)$. 

\begin{figure}
    \centering

\begin{tikzpicture}[
  every node/.style={circle, draw, fill=white, inner sep=2pt, minimum size=22pt},
  scale=1
]

% Root
\node (r) at (0,0) {1};

% Level 1 (distance 1)
\foreach \i/\angle in {0/90,1/210,2/330} {
  \node (a\i) at (\angle:1.5) {$a_\i$};
  \draw (r) -- node[
      midway,
      shift=(\angle-90:0.2),
      draw=none,
      fill=none,
      inner sep=0pt
    ] {$\i$} (a\i);
}

% Level 2 (distance 2)
\foreach \i/\angle in {0/90,1/210,2/330} {
  \pgfmathtruncatemacro{\label}{mod(\i+1,3)}
  \pgfmathtruncatemacro{\labell}{mod(\i+2,3)}
  \node[fill = blue!30] (b\i1) at ({\angle+20}:3) {\scriptsize{$a_\i a_\label$}};
  \node[fill = blue!30] (b\i2) at ({\angle-20}:3) {\scriptsize{$a_\i a_\labell$}};
  \draw (a\i) -- node[
      midway,
      shift=(\angle+110:0.2),
      draw=none,
      fill=none,
      inner sep=0pt
    ] {$\label$} (b\i1);
  \draw (a\i) -- node[
      midway,
      shift=(\angle+70:0.2),
      draw=none,
      fill=none,
      inner sep=0pt
    ] {$\labell$} (b\i2);
}

\node[fill=red!30] at (b21) {\scriptsize{$a_2a_0$}};

\end{tikzpicture}
    \caption{The 2-ball around the identity in $\text{Cay}(\Z_2^{*3},S_3)$ with our edge labels. In the games $G(x,a_2a_0)$, Alice is responsible for labeling the red node (and all the other nodes in that sextant), Bob for the blue nodes (and all the other nodes in those five sextants), and the white nodes other than the identity are given the label 0. }
    \label{fig:game}
\end{figure}
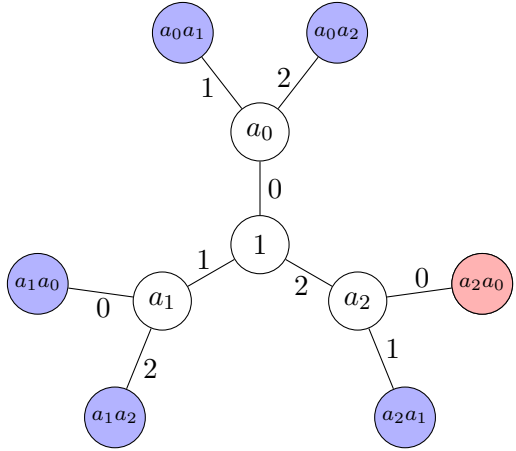

It is time to show that there is no Borel homomorphism from $G$ to $H$. Suppose $c$ is one. For $r \in \N$, $S_d^{\leq r}$ denotes the set of elements of $\Z_2^{*d}$ with word length at most $r$. I.e, those in the radius $r$-ball around in $\textnormal{Cay}(\Z_2^{*d},S_d)$. Let $T = S_d^{\leq l+1} \setminus S_d^{\leq l}$, the set of elements with word length exactly $l+1$. Note that $|T| = d \cdot (d-1)^l$. Therefore, by adding isolated vertices if necessary, we may identify $V(H)$ with $T$. 

For each $x \in V(\mathcal{H})$ and $\tau \in T$, we define a game, call it $G(x,\tau)$, in which two players, Alice and Bob, define one by one the entries in an element $y \in \omega^{\Z_2^{*d} \setminus \{1\}}$. We set $y(\gamma) = 0$ for $\gamma \in S_d^{\leq l} \setminus \{1\}$. Then Alice chooses the value of $y(\gamma)$ on all words $\gamma$ which begin with $\tau$ and Bob chooses the value of $y$ on all other elements. The order of play is that for each $k > l$ in turn, Alice first chooses a value for words of length $k$ beginning with $\tau$, then Bob chooses a value for all other words of length $k$. This is illustrated in Figure \ref{fig:game}. 
In the end we obtain a vertex $\Phi(x,y) \in V(G)$. Bob wins if and only if $c(\Phi(x,y)) = \tau$. 

\begin{claim}
    For each $x \in V(\mathcal{H})$, there is some $\tau \in T$ for which Alice does not have a winning strategy in $G(x,\tau)$.
\end{claim}

\begin{proof}
    If not, then by playing the winning Alice strategies in $G(x,\tau)$ for all $\tau \in T$ against eachother, we obtain a $y$ for which $c(\Phi(x,y)) \neq \tau$ for all $\tau \in T$. But $T$ is the codomain of $c$. 
\end{proof}

Since $c$ and $\Phi$ are Borel, the Borel determinacy theorem implies that for each $x$, there is some $\tau \in T$ for which Bob has a winning strategy in $G(x,\tau)$. Let $d(x) \in T$ be the lex-least such $\tau$. 
$d$ may not be Borel, but it is measurable with respect to the $\sigma$-algebra of so-called Game-Borel sets. In particular it is Baire measurable. See \cite[Proposition 1.3]{hom_graph}.

Now choose $\tau \in T$ for which $d\inv(\tau)$ is nonmeager. Write $\tau = a_{i_0} \cdots a_{i_l}$, and let $\sigma = a_{i_0} \cdots a_{i_l} \cdots a_{i_0} = \tau a_{i_l} \tau\inv$. 

\begin{claim}
    There exist $x,x' \in d\inv(\tau)$ so that, letting $y = f_{i_{l-1},0}( \cdots f_{i_0,0}(x) \cdots)$ and likewise for $y'$ and $x'$, $(y,y') \in E(\mathcal{H})$ and the edge between them has label $i_l$. 
\end{claim}
\begin{proof}
    Let $g = f_{i_{l-1},0} \circ \cdots \circ f_{i_0,0}$. By Claim \ref{c:continuous_inv} $g$ is a homeomorphism of $V(\mathcal{H})$.  
    Thus $g(d\inv(\tau))$ is Baire measurable and nonmeager, so we can apply Lemma \ref{lem:edge_G0}. 
\end{proof}

\begin{figure}
    \centering

\begin{tikzpicture}[
  every node/.style={circle, draw, fill=white, inner sep=2pt, minimum size=22pt},
  scale=1
]

% Root
\node (r) at (0,0) {1};

% Level 1 (distance 1)
\foreach \i/\angle in {0/90,1/210,2/330} {
  \node (a\i) at (\angle:1.5) {};
  \draw (r) -- node[
      midway,
      shift=(\angle-90:0.2),
      draw=none,
      fill=none,
      inner sep=0pt
    ] {$\i$} (a\i);
}

% Level 2 (distance 2)
\foreach \i/\angle in {0/90,1/210,2/330} {
  \pgfmathtruncatemacro{\label}{mod(\i+1,3)}
  \pgfmathtruncatemacro{\labell}{mod(\i+2,3)}

  \ifnum\i=2
    %%%
  \else
  \node[fill = blue!30] (b\i1) at ({\angle+20}:3) {};
    \draw (a\i) -- node[
      midway,
      shift=(\angle+110:0.2),
      draw=none,
      fill=none,
      inner sep=0pt
    ] {$\label$} (b\i1);
  \fi
  
  \node[fill = blue!30] (b\i2) at ({\angle-20}:3) {};

  \draw (a\i) -- node[
      midway,
      shift=(\angle+70:0.2),
      draw=none,
      fill=none,
      inner sep=0pt
    ] {$\labell$} (b\i2);
}

\begin{scope}[shift = {(6,0)}]
\node (2r) at (0,0) {$\sigma$};

% Level 1 (distance 1)
\foreach \i/\angle in {0/90,1/330,2/210} {
  \node (2a\i) at (\angle:1.5) {};
  \draw (2r) -- node[
      midway,
      shift=(\angle+90:0.2),
      draw=none,
      fill=none,
      inner sep=0pt
    ] {$\i$} (2a\i);
}

% Level 2 (distance 2)
\foreach \i/\angle in {0/90,1/330,2/210} {
  \pgfmathtruncatemacro{\label}{mod(\i+1,3)}
  \pgfmathtruncatemacro{\labell}{mod(\i+2,3)}
  \ifnum\i=2
    %%%
  \else
  \node[fill = teal!30] (2b\i1) at ({\angle-20}:3) {};
    \draw (2a\i) -- node[
      midway,
      shift=(\angle-110:0.2),
      draw=none,
      fill=none,
      inner sep=0pt
    ] {$\label$} (2b\i1);
  \fi
  \node[fill = teal!30] (2b\i2) at ({\angle+20}:3) {};

  \draw (2a\i) -- node[
      midway,
      shift=(\angle-70:0.2),
      draw=none,
      fill=none,
      inner sep=0pt
    ] {$\labell$} (2b\i2);
}
\end{scope}

\draw (a2) -- node[midway, yshift=6pt, xshift=-10pt, fill = none, draw=none] {0} (2a2);
\node at (a2) {};
\node at (2a2) {$\tau$};
\draw[line width=2pt, dotted] (3,3.5) -- (3,-3);

\end{tikzpicture}
    \caption{Playing two copies of Bob against eachother. As in Figure \ref{fig:game}, $d = 3$, $l = 1$, and $\tau = a_2a_0$. 
    The two Bobs agree on edge labels, but
    vertices are labeled according to the unprimed Bob's perspective. 
    The unprimed Bob is responsible for labeling the blue vertices. 
    From his perspective, everything to the right of the dotted line is being labeled by Alice. In reality we label $\tau$ with $m$ and the remaining nonidentity white nodes 0, while the primed Bob labels the green vertices.
    The fact that each $f_{i,j}$ is an involution can be used to see that the primed Bob agrees with the unprimed Bob, modulo shifting by $\sigma$, about which element of $V(G)$ is being constructed. }
    \label{fig:two_bob}
\end{figure}
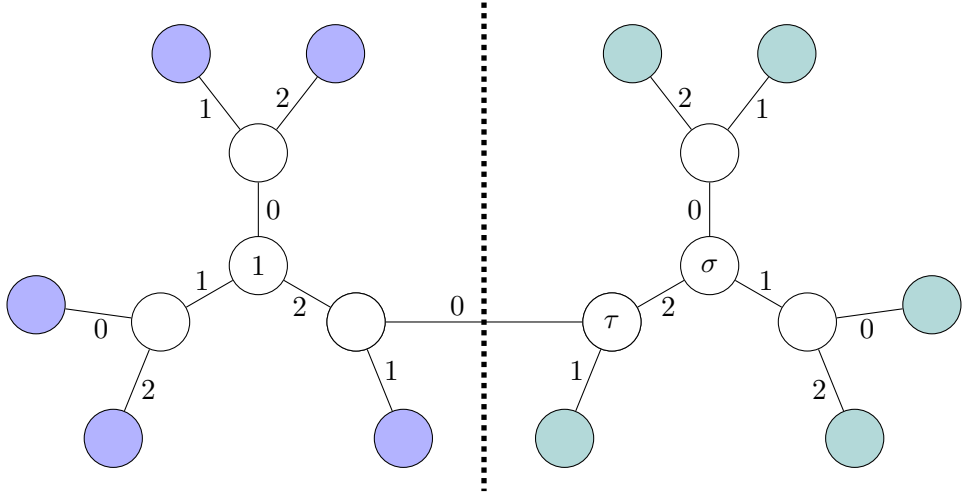

Fix such $x,x',y,y'$ and let $m \in \omega$ such that $f_{i_l,m}(y) = y'$. Let $B$ be a winning strategy for Bob in $G(x,\tau)$ and likewise for $B'$ and $x'$. We can now play these strategies against each other as explained in Figure \ref{fig:two_bob}. Let $z,z' \in V(G)$ be the vertices produced from $B$ and $B'$'s perspective respectively. Then we have $c(z)=c(z') = \tau$ and $\sigma \cdot z = z'$. For $ j \leq 2l+1$ let $\sigma_j \in \Z_2^{*d}$ be the word consisting of the $j$ rightmost characters in $\sigma$, so that $\langle \sigma_j z \mid j \leq 2l+1 \rangle$ is the path in $G$ from $z$ to $z'$. Then $\langle c(\sigma_j z) \mid j \leq 2j+1 \rangle$ is a circuit in $H$ of length $2l+1$, (it starts and ends at $\tau$) contradicting the odd girth assumption on $H$. This completes the proof of Theorem \ref{thm:hyperfinite_girth}. 

\section{Further questions}

We end the paper with two more questions. In the previous section we showed that for $k$ sufficiently large and $\epsilon$ sufficiently small, a hyperfinite Borel 3-regular forest may fail to admit a Borel homomorphism to the Schrijver graph $K'((2+\epsilon)k,k)$. 
Whether this holds for the full Kneser graph $K((2+\epsilon)k,k)$ is unknown. 

\begin{question}
    Let $G$ be a hyperfinite Borel  3-regular forest. Is $\chi_B^*(G) = 2$?
\end{question}

It follows from the remarks at the end of the introduction that the answer is yes if one ignores null sets. 

Finally, we would like to advertise the following test question which does not involve hyperfiniteness. For the most part in this paper we have been interested in the Kneser graphs $K(n,k)$ asymptotically in $n$ and $k$, but our knowledge about particular small values is limited. 

\begin{question}
    Let $G$ be a Borel 3-regular forest. Does $G$ admit a Borel homomorphism to $K(6,2)$?
\end{question}

The answer is again yes if one ignores null sets since any such $G$ has measurable 3-colorings by \cite{conley2016brooks}. 
Recall also that by Proposition \ref{prop:meas_ind_eq_Td} any such $G$ has a Borel homomorphism to $K(3k,k)$ for $k$ large enough. Furthermore, the answer is negative for $K(5,2)$ since $\chi(K(5,2)) = 3$ by Theorem \ref{thm:Lovasz} while $\chi_B(G)$ may be 4 \cite{Marks}, and the answer is positive for $K(7,2)$ as we quickly sketch below.

\begin{prop}
    Let $d \geq 2$, $k \in \N$, and $G$ a $d$-regular Borel forest. $G$ admits a Borel homomorphism to $K(dk+1,k)$. 
\end{prop}

That is, $G$ admits a Borel $k$-fold $(dk+1)$-coloring.
Note that in the case $k = 1$ this is just the fact that $\chi_B(G) \leq d+1$ \cite{kst}. 

\begin{proof}
    Our approach is based on the ``$d$-coloring off lines'' problem from \cite{BCGGRV}. First note that we can find a Borel partial proper $d$-coloring $c : V(G) \rightharpoonup d$ with the property that, letting $X = V(G) \setminus \dom(c)$, $G \res X$ is 2-regular, and for all $y \in \dom(c)$ adjacent to a point in $X$, $c(y) < d-2$. To do this, by \cite{kst} iteratively remove Borel maximal independent sets from $V(G)$ and have these be the $c\inv(i)$ for $i < d-2$. Letting $Y$ be the remaining set of vertices, note $G \res Y$ has maximum degree 2. The union of the connected components of $G \res Y$ which are not 2-regular can then easily be Borel 2-colored using $d-1$ and $d-2$. 

    We now describe our $k$-fold $(dk+1)$-coloring. First, for each $i < d$, Each point $x \in c\inv(i)$ will simply get the colors $ik, ik+1, \ldots, ik+(k-1)$. 

    Each point in $X$ now has at least the colors greater than equal to $(d-2)k$ available. Thus, it suffices to give a Borel $k$-fold $(2k+1)$-coloring of $G \res X$. Since $K(2k+1,k)$ has an odd cycle by Lemma \ref{lem:kneser_odd_girth}, this is possible \cite{GR_paths}.
\end{proof}

\section*{Acknowledgments}

The author is supported by the NSF under award number DMS-2402064.

\bibliographystyle{amsalpha}
\bibliography{main}

\end{document}